\documentclass[12pt]{amsart}
\usepackage{amscd}
\usepackage{verbatim}
\usepackage{amssymb, amsmath, amsthm, amscd,ifthen}
\usepackage[dvips]{graphics}
\usepackage[cp866]{inputenc}
\usepackage{graphicx}
\usepackage{epsfig}
\usepackage{color}

\usepackage{amsmath,amssymb,amscd,}
\usepackage[all]{xy}

\usepackage{amssymb}

\newtheorem{thm}{Theorem}[section]

\newtheorem{prop}[thm]{Proposition}

\newtheorem{lemma}[thm]{Lemma}
\newtheorem{cor}[thm]{Corollary}

\theoremstyle{definition}
\newtheorem{rem}[thm]{Remark}
\newtheorem{dfn}[thm]{Definition}

\title{A Tverberg type theorem for \\ collectively unavoidable complexes}

\author[D. Joji\'{c}]{Du\v{s}ko Joji\'{c}}
\author[G. Panina]{Gaiane Panina}
\author[R. \v{Z}ivaljevi\'{c}]{Rade \v{Z}ivaljevi\'{c}}

\address[D. Joji\'{c}]{ Faculty of Science, University of Banja Luka}
\address[G. Panina]{Mathematics \& Mechanics Department, St. Petersburg State University}
\address[R. \v{Z}ivaljevi\'{c}]{Mathematical Institute SASA, Belgrade}

\address{
  }


\begin{document}

 \maketitle \setcounter{section}{0}

\begin{abstract}
We prove (Theorem~\ref{thm:main}) that the symmetrized deleted join $SymmDelJoin(\mathcal{K})$ of
a ``balanced family'' $\mathcal{K} = \langle K_i\rangle_{i=1}^r$ of collectively
$r$-unavoidable subcomplexes of $2^{[m]}$ is $(m-r-1)$-connected. As a consequence we
obtain a Tverberg-Van Kampen-Flores type result (Theorem~\ref{thm:main-2}) which is more conceptual and
more general then previously known results.  Already the case $r=2$ of Theorem~\ref{thm:main-2}
seems to be new as an extension of the classical Van Kampen-Flores theorem. The main tool used in the paper is
R. Forman's discrete Morse theory.
\end{abstract}

\section{Introduction}

Tverberg-Van Kampen-Flores type results have been for many years one of the
central research themes in topological combinatorics. The last decade has been
particularly fruitful with some long standing conjectures resolved, as summarized  by several review papers \cite{bbz, lgmm, sko, Z17}, covering
different aspects of the theory.


Certainly the most striking among the new results is the resolution (in the negative!) of the general ``Topological Tverberg Problem''
\cite{MaWa, frick, bfz-1}. On the positive side is the proof \cite[Theorem 1.2]{jvz-2} of the ``Balanced Van Kampen Flores theorem'' indicating in which direction one
   can expect new positive results.

   \medskip
In this paper we prove a result (Theorem~\ref{thm:main-2}) which we see as a candidate for the currently most general and far reaching
result of Van Kampen-Flores type. Indeed, this result contains the ``Balanced Van Kampen-Flores theorem'' as a special
case (Corollary~\ref{cor:main2glavna}), as well as other results of this type. Note that already the case $r=2$ of the theorem (see Section~\ref{sec:example}),
which extends the classical Van Kampen-Flores theorem, doesn't seem to have been recorded before.

\medskip
Surprisingly enough Theorem~\ref{thm:main-2} is not only more general but it also provides a more conceptual and possibly more elegant and transparent approach.
The new approach relies on the concepts of ``collectively unavoidable'' (Definition~\ref{dfn:pigeonhole}) and
``balanced'' (Definition~\ref{dfn:balanced}) $r$-tuples
$\mathcal{K} = \langle K_i\rangle_{i=1}^r = \langle K_1,\dots, K_r\rangle$ of simplicial complexes.
Recall that collectively unavoidable complexes were originally introduced and studied in \cite{jnpz} as a common
generalization of pairs $\langle K, K^\circ\rangle$ of Alexander dual complexes (Alexander
$2$-tuples) and $r$-unavoidable complexes of \cite{bfz} and \cite{bestiary}.

\medskip
The key topological (connectivity) property of `balanced' collectively $r$-unavoidable $r$-tuples $\mathcal{K}$ is proved in Theorem~\ref{thm:main}.
The proof is based on discrete Morse theory and the construction of the discrete Morse function is particularly well adapted for
applications to deleted joins and symmetrized deleted joins of complexes.

\medskip
Here is a brief outline of the paper.
In Section~\ref{sec:connectivity} we prove the first main result of the paper, Theorem~\ref{thm:main}, which estimates the connectivity of the symmetrized
deleted join of an r-tuple of complexes, under assumption that the $r$-tuple is both balanced and unavoidable. The main tool in the proof is the R. Forman's
discrete Morse theory (DMT).
The new Van Kampen-Flores type result (Theorem~\ref{thm:main-2}) is obtained in Section~\ref{sec:Tverberg} as a corollary of Theorem~\ref{thm:main}.
The proof uses the usual \emph{Configuration Space/Test Map Scheme}, see \cite{jvz-2},  and relies on Volovikov's version of Borsuk-Ulam theorem \cite{Z17}.
In Section~\ref{sec:collective} we discuss criteria for an r-tuple of simplicial complexes to be  both balanced and collectively unavoidable.
Finally,  for the reader's convenience, we outline in the Appendix (Section~\ref{sec:collective}) basic principles of discrete Morse theory  \cite{Forman1, Forman2}.

\medskip
\subsection*{Acknowledgements}It is our pleasure to acknowledge the support and hospitality
of the {\em  Centre International de Rencontres Math\' ematiques} (CIRM, Marseille), where
in the fall of 2018  this paper was initiated as a `research in
pairs' project.  R. \v{Z}ivaljevi\'{c}
acknowledges the support of the Ministry of Education, Science and
Technological Development of Republic of Serbia, Grant 174034.

\section{Connectivity of the symmetrized deleted join}\label{sec:connectivity}

\subsection{Preliminary definitions}

\begin{dfn}\label{dfn:balanced}
We say that a simplicial complex $K\subseteq 2^{[m]}$ is {\em $(m,k)$-balanced} if it is
positioned between two consecutive skeleta of a simplex on $m$ vertices,
\begin{equation}\label{eqn:binoms}
   {{[m]}\choose{\leqslant k}} \subseteq K \subseteq {{[m]}\choose{\leqslant k+1}}\, .
\end{equation}
\end{dfn}

\begin{dfn}
The {\em deleted join} \cite[Section~6]{Mat} of a family
$\mathcal{K} = \langle K_i\rangle_{i=1}^r = \langle K_1,\dots,
K_r\rangle$ of subcomplexes of $2^{[m]}$ is the complex
$\mathcal{K}^\ast_\Delta = K_1\ast_\Delta\dots \ast_\Delta  K_r
\subseteq (2^{[m]})^{\ast r}$ where $A = A_1\sqcup\dots\sqcup A_r\in
\mathcal{K}^\ast_\Delta$ if and only if $A_j$ are pairwise
disjoint and $A_i\in K_i$ for each $i=1,\dots, r$.

The {\em symmetrized deleted join} of  $\mathcal{K}$ is defined as
\[
SymmDelJoin(\mathcal{K}) := \bigcup_{\pi\in S_r} K_{\pi(1)}\ast_\Delta\dots \ast_\Delta  K_{\pi(r)}\subseteq (2^{[m]})^{\ast r} \, .
\]
where the union is over the set of all permutations of $r$ elements. (Here $S_r$ stands for the symmetric group.)
\end{dfn}
An element  $A_1\sqcup\dots\sqcup A_r \in SymmDelJoin(\mathcal{K})$ is from here on recorded as
$(A_1,A_2,...,A_r;B)$ where $B$ is the complement of $\cup_{i=1}^r~A_i$, so in particular $A_1\sqcup\dots\sqcup A_r\sqcup B = [m]$ is a partition of $[m]$.
Note that the dimension of the simplex can be read of from $\vert B\vert$ as ${\rm dim}(A_1,...,A_r;B) = m - \vert B\vert -1$.

\medskip
Collectively unavoidable $r$-tuples of complexes are  introduced in \cite{jnpz}.
They were originally studied as a common
generalization of pairs of Alexander dual complexes, Tverberg unavoidable complexes of \cite{bfz} and $r$-unavoidable
complexes from \cite{bestiary}.

\begin{dfn}\label{dfn:pigeonhole}
An ordered $r$-tuple $\mathcal{K} = \langle K_1,...,K_r\rangle$ of
subcomplexes of $2^{[n]}$ is {\em collectively
$r$-unavo\-id\-able} if for each ordered collection
$(A_{1},...,A_{r})$ of disjoint sets in $[n]$ there exists $i$
such that $A_i\in K_i$.
\end{dfn}

\subsection{The main theorem}

\begin{thm}\label{thm:main} Suppose that $\mathcal{K} = \langle K_i\rangle_{i=1}^r = \langle K_1,\dots,
K_r\rangle$ is a {\em collectively unavoidable} family of subcomplexes of $2^{[m]}$. Moreover, we assume that there exists $k$ such that $K_i$ is $(m,k)$-balanced
for each $i=1,\dots, r$. Then the associated symmetrized  deleted join
$$SymmDelJoin(\mathcal{K}) = SymmDelJoin(K_1,\dots, K_r)$$
is $(m-r-1)$-connected.
\end{thm}

\proof{(outline)}  In Section~\ref{sec:construction} we construct a discrete Morse function on
the symmetrized deleted join  $SymmDelJoin(\mathcal{K})$. In other words we describe an {\em acyclic matching} of simplices in
$SymmDelJoin(\mathcal{K})$  (see the Appendix for a brief description of this technique). The proof of the acyclicity is given in
Section~\ref{sec:acyclicity}. Following one of the central principles of Discrete Morse Theory, the complex $SymmDelJoin(\mathcal{K})$ is homotopy equivalent
to a complex built from critical simplices. So the proof is concluded (Section~\ref{sec:critical}) by showing that the dimension of all critical simplices is at least
(m-r) (with the exception of the unique simplex of dimension $0$).

\subsection{Construction of a discrete Morse function}\label{sec:construction}

\hfill\break
Assume that $A_1\sqcup\dots\sqcup A_r \sqcup B = [m]$ is an ordered partition, interpreted as a simplex
$(A_1,A_2,...,A_r;B)\in (2^{[m]})^{\ast r}$.
A simplex labeled by $(A_1,...,A_r;B)$ is called \textit{large} if $|B|\leq r-1$. The dimension of a large simplex
is at least $m-r$.

\medskip

Our aim is to construct a Discrete Morse Function (DMF) such that all simplices that are not large are matched
 (with one $0$-dimensional exception).
This is precisely the condition needed for the $(m-r-1)$-connectivity, see the Appendix.

\bigskip

\textbf{Step 1.}
Set $a_1:= min(B,A_1)$ and match the simplices

$(A_1,...,A_r;Ba_1)$ and $(A_1a_1,...,A_r;B)$  whenever  both of them are elements of the complex $SymmDelJoin(\mathcal{K})$.

\bigskip

Let us analyze non-matched simplices.

There is exactly one $0$-dimensional unmatched simplex, namely the simplex $(\{1\},\emptyset,\dots,\emptyset; [m]\setminus\{1\})$.
The remaining unmatched simplices  are  some of the simplices of the form $(A_1,...,A_r;Ba_1)$.

For bookkeeping purposes these simplices are recorded as \textit{Step 1 -- Type 1}  unmatched simplices.

\bigskip

\textbf{Step 2.}
Set $a_2:= min((B\cup A_2)\setminus [1,a_1])$ and match the simplices

$(A_1,...,A_r;Ba_2)$ and $(A_1,A_2a_2,...,A_r;B)$  whenever
\begin{enumerate}
  \item  both of them belong to  $SymmDelJoin$,
  \item both of them were not matched before (that is, on step 1).
\end{enumerate}

Unmatched simplices (we ignore now the  zero-dimensional one) are now of three types:
\begin{enumerate}
  \item $(A_1,...,A_r;Ba_2)$.

  We say that this is \textit{Step 2 -- Type 1 simplex}.

  \item $(A_1,A_2a_2,...,A_r;B)$,

We say that this is \textit{Step 2 -- Type 2 simplex}.

  \item  those with  $(B\cup A_2)\setminus [1,a_1]=\emptyset$. All the simplices of this type are large.
\end{enumerate}

\bigskip

Other steps go analogously so for example the Step k looks as follows:

\medskip
\textbf{Step k.}
Set $a_k:= min((B\cup A_k)\setminus [1,a_{k-1}])$ and match the simplices

$(A_1,...,A_r;Ba_k)$ and $(A_1,\dots, A_ka_k,\dots, A_r;B)$  whenever
\begin{enumerate}
  \item  both of them belong to  $SymmDelJoin$,
  \item both of them were not matched before (that is, on step $\leq k-1$).
\end{enumerate}

Unmatched simplices  are  of the three types:
\begin{enumerate}
  \item $(A_1,...,A_r;Ba_k)$.

  We say that this is \textit{Step k -- Type 1 simplex}.

  \item $(A_1,\dots, A_ka_k,\dots, A_r;B)$,

We say that this is \textit{Step k -- Type 2 simplex}.

  \item  those with  $(B\cup A_k)\setminus [1,a_{k-1}]=\emptyset$. All the simplices of this type are large.
\end{enumerate}

\medskip
We {proceed analogously} for $k=1,\dots, r$.

\medskip

This completes the construction of a discrete vector field on the symmetrized deleted join $SymmDelJoin(\mathcal{K})$.
It remains to be shown that this discrete vector field is acyclic.

\bigskip
\subsection{The acyclicity of the discrete vector field}\label{sec:acyclicity}

\begin{dfn}
Given a simplex $\sigma = (A_1,\dots, A_r; B)\in SymDelJoin$, its {\em passport} $p(\sigma) = (a_1,\dots, a_r)$ is defined as
$a_i := \min ((A_i\cup B)\setminus [1,a_{i-1}])$, provided the indicated set is non-empty. Otherwise we set $a_i:= \infty$.
\end{dfn}
We tacitly assume that the passports are linearly ordered by the lexicographic ordering.

\medskip\noindent
{\bf Claim~1.} Along a gradient path the passport does not increase.  {\qed}

\medskip
From here we immediately conclude that if there is a closed path, the passport is constant.

\medskip
For the next claim recall that a {\em migrating element} (see the Appendix), corresponding to the ``splitting step'' $\beta^{p+1}_i \searrow \alpha^p_{i+1}$ in a gradient path,
is the vertex $v\in \beta^{p+1}_i \setminus  \alpha^p_{i+1}$. Similarly in the ``matching step''  $\alpha^p_i \nearrow \beta^{p+1}_i$ it is the element $v\in \beta^{p+1}_i \setminus  \alpha^p_{i}$.
For illustration the {matching} step $(A_1,...,A_r;Ba_k) \nearrow (A_1,\dots, A_ka_k,\dots, A_r;B)$ can be described as a migration of $a_k$ from
$B' = Ba_k$ to $A_k' = A_ka_k$. Similarly, in the splitting step $(A_1,\dots, A_k\nu,\dots, A_r;B) \searrow (A_1,\dots, A_k,\dots, A_r;B\nu)$, the
element $\nu$ migrates from $A_k' = A_k\nu$ to $B\nu$.

\medskip\noindent
{\bf Claim~2.} Assuming that there is a closed path, the migrating elements can come only from the set $\{a_i\}_{i=1}^r$.

\medskip
Indeed, if $a_i$  migrates from $A_i$ to $B$ than it can come back only as an element of $A_i$.  {\qed}

\medskip In summary, we conclude that a closed path, if it exists, is uniquely determined by the sequence of indices of migrating elements.

\medskip
For instance, a fragment of a closed path, producing indices $$ i_1=3, i_2=4, i_3 = 2, $$ looks exactly as follows,

\[ \begin{array}{c}
(A_1,A_2, A_3, A_4a_4; Ba_1a_3a_2) \\
\downarrow 3  \\
(A_1,A_2, A_3a_3, A_4a_4; Ba_1a_2) \\
\downarrow 4 \\
(A_1,A_2, A_3a_3, A_4a_4; Ba_1a_2a_4) \\
\downarrow 2 \\
(A_1,A_2a_2, A_3a_3, A_4; Ba_1a_4) \\
\end{array}\]

  Recall that in a gradient path we distinguish the ``matching steps'' (as the steps when some $a_i$ migrates from $B$ to $A_i$) from the
  ``splitting steps'' (when some $a_i$ migrates from $A_i$ to $B$).

  \medskip
Note that each of the migrating elements $a_i$ participates in an equal number of matching and splitting steps.

\medskip
Assuming that the indexing of steps in the closed path is chosen so that the even steps correspond to the ``matching steps'' (and the
odd steps are the ``splitting steps'') then the indices satisfy the following relation:

\begin{equation}\label{eqn:indices}
\begin{array}{l}
(\forall j)\, i_j\neq i_{j+1} \\
(\forall j)\, i_{2j+1} > i_{2j+2}.
\end{array}
\end{equation}

Taking the minimal migrating index $i$ leads to a contradiction with the second inequality in  (\ref{eqn:indices}).

\subsection{Critical simplices are large}\label{sec:critical}

Let $\Phi(A_1,A_2,...,A_r;B)$ be the set of all permutations
$\phi\in S_r$  such that  $A_i \in K_{\phi(i)}$ for each $i=1,\dots, r$.
Clearly,
$(A_1,A_2,...,A_r;B)$ belongs to $SymmDelJoin(\mathcal{K})$  iff   $\Phi(A_1,A_2,...,A_r;B)$ is non-empty.

Let us look at the non-matched simplices after the last step, that is, after the step $r$.

 We need to show that each non-matched simplex  is {\em large} (except for a single $0$-dimensional simplex).

Let us assume that the simplex $(A_1,...,A_r;B)$ is unmatched. Let $I\subset [r]$ be the set of all indices such that
$$
  i\in I  \quad \Leftrightarrow \quad (A_1,\dots, A_r; B) \mbox{ {\rm is \textit{Type 1 on Step $i$}} } \, .
$$
Then $k\notin I$ implies that the simplex $(A_1,\dots, A_r; B)$ is \textit{Type 2 on Step $k$}.

Assume now that the simplex is not large, that is, $|B|>r-1$.
Choose a permutation $\phi \in \Phi(A_1,...,A_r;B)$. If
$k\in I$ then  $a_k\in B$, and $A_ka_k\notin K_{\phi(k)}$. If
$k \notin I$ then $|A_k|=\nu +1$, so $A_k$ plus any other element $b$ is no longer in $K_{\phi(k)}$.

Now we enlarge each $A_i$ by an element $\pi(i)\in B$. More precisely let $A_i' := A_i\cup\{\pi(i)\}$ where
for $i\in I$ we define $\pi(i) = a_i$ and if $i\notin I$ then $\pi(i)= b_i\in B\setminus\{a_i\}$. By construction $A_i'\notin K_{\pi(i)}$
and, since $A'_1\sqcup\dots\sqcup A'_r = [r]$ is a partition, we obtain a contradiction with the collective unavoidability of the family
$\mathcal{K} = \langle K_i \rangle_{i=1}^r$.
\qed

\section{A general Tverberg-Van Kampen-Flores theorem for balanced complexes}
\label{sec:Tverberg}

The following theorem of Tverberg-Van Kampen-Flores type is the main result of \cite{jvz-2}.
It is very likely the most general known result that evolved from the classical Van Kampen-Flores theorem \cite[Section 22.4.3]{Z17}.
For example it extends and contains as a special case the `Generalized Van Kampen-Flores Theorem' of Sarkaria \cite{Sar91}, Volovikov \cite{Vol96-2}, and
Blagojevi\' c, Frick and Ziegler \cite{bfz}.

\begin{thm}\label{thm:glavna-jvz} {\rm (\cite[Theorem 1.2]{jvz-2})}
Let $r\geq 2$ be a prime power, $d \geq 1$, $N \geq (r - 1)(d +
2)$, and $rk+s \geq (r-1)d$ for integers $k \geq 0$ and $0 \leq s
< r$. Then for every continuous map $f : \Delta_N \rightarrow
\mathbb{R}^d$, there are $r$ pairwise disjoint faces
$\sigma_1,\ldots,\sigma_r$ of $\Delta_N$ such that
$f(\sigma_1)\cap \cdots\cap f(\sigma_r) \neq \emptyset$, with
$\textrm{dim }\sigma_i\leq k + 1$ for $1 \leq i \leq s$ and
$\textrm{dim }\sigma_i\leq k $ for $s < i \leq r$.
\end{thm}

Theorem~\ref{thm:glavna-jvz} confirmed the conjecture of
Blagojevi\' c, Frick, and Ziegler about the existence of `balanced
Tverberg partitions' (Conjecture~6.6 in \cite{bfz}).
Among the consequences of this theorem is a positive answer
(see \cite[Theorem~7.2]{jvz-2}) to the `balanced case' of the problem
whether each {\em admissible} $r$-tuple is {\em Tverberg
prescribable}, \cite[Question~6.9]{bfz}.

\medskip
The term `balanced partitions' in both results refers to the constraint that a Tverberg $r$-tuple
$(\sigma_1, \sigma_2,\dots,\sigma_r)$  is sought in the symmetric deleted join
\begin{equation}
SymmDelJoin(K_1,\dots, K_r)
\end{equation}
of adjacent skeleta  of the simplex $\Delta_N = 2^{[N+1]}$,
\begin{equation}\label{eq:balance-skeleta}
K_1 = \dots = K_s = {[N+1]\choose \leqslant k+2}, \quad  K_{s+1} = \dots = K_r = {[N+1]\choose \leqslant k+1} \, .
\end{equation}
It is known \cite{jnpz} that the collection of subcomplexes of $2^{[m]}$,
\begin{equation}\label{eq:binom-skel}
\left({[m]\choose \leqslant m_1},\dots,{[m]\choose \leqslant
m_r}\right)
\end{equation}
is always a collectively $r$-unavoidable, provided
$m=\sum_{i=1}^r m_i+r-1$ and in particular (\ref{eq:balance-skeleta}) is such a collection if $N+1 = s(k+2) + (r-s)(k+1) + r-1$.

\medskip
Conditions (\ref{eq:balance-skeleta}) and (\ref{eq:binom-skel}) indicate that collectively $r$-unavoidable complexes
behave very well if in addition we assume that they are balanced. This is precisely the content of Theorem~\ref{thm:main}.
From here it is not difficult to derive a general theorem of Van Kampen-Flores type which includes Theorem~\ref{thm:glavna-jvz} as
a special case and which seems to be new already in the case $r=2$ (Section~\ref{sec:example}).

\begin{thm}\label{thm:main-2}
Suppose that $\mathcal{K} = \langle K_i\rangle_{i=1}^r = \langle K_1,\dots,
K_r\rangle$ is a {\em collectively unavoidable} family of subcomplexes of $2^{[m]}$, where $r = p^\nu$ is a power of a prime number. Assume that $K_i$ is $(m,k)$-balanced
for each $i=1,\dots, r$. Suppose that $N\geq (r-1)(d+2)$ and let $m = N+1$. Then for each continuous map $f : \Delta_N \rightarrow \mathbb{R}^d$,
from an $N$-dimensional simplex into a  $d$-dimensional euclidean space, there exist vertex-disjoint faces $\sigma_1,\dots, \sigma_r$ of $\Delta_N$ such that
$f(\sigma_1)\cap\dots\cap f(\sigma_r)\neq\emptyset$ and
\[
     \sigma_1\in K_1, \sigma_2\in K_2, \dots, \sigma_r\in K_r \, .
\]
\end{thm}

\begin{proof}
Let $V = ((\mathbb{R}^d)^{\ast r}) \cong \mathbb{R}^{rd + r-1}$ be the $r$-fold join of $\mathbb{R}^d$ and let $D \cong \mathbb{R}^d$ be the diagonal
subspace in $V$.
A map $f: \Delta_N \rightarrow \mathbb{R}^d$ induces an $S_r$-equivariant map
\begin{equation}\label{eq:equivariant}
  F : SymmDelJoin(\mathcal{K}) \rightarrow  V/D \cong \mathbb{R}^{(r-1)(d+1)}
\end{equation}
and the theorem follows from the observation that this map must have a zero. If not, there arises an equivariant map
\begin{equation}\label{eq:equivariant-2}
  F' : SymmDelJoin(\mathcal{K}) \rightarrow  S(V/D) \cong S^{(r-1)(d+1)-1}
\end{equation}
which contradicts Volovikov's theorem \cite{Z17}, since the space $SymmDelJoin(\mathcal{K})$ is $(m-r-1)$-connected (Theorem~\ref{thm:main})
and  $m-r-1 \geq (r-1)(d+1)-1$ is equivalent to $N \geq (r-1)(d+2)$.
\end{proof}

\begin{cor}\label{cor:main2glavna}
Theorem~\ref{thm:glavna-jvz} is a special case of Theorem~\ref{thm:main-2}.
\end{cor}

\begin{proof}
It can be easily shown that Theorem~\ref{thm:glavna-jvz} is reduced to the case when $N = (r-1)(d+2)$ and $rk+s = (r-1)d$. Indeed, for given $r$ and $d$
one is interested in the smallest $k$ and $N$ for which the theorem is still valid.

All that remains to be checked is that the collection $\mathcal{K}$, described by  (\ref{eq:balance-skeleta}), is collectively $r$-unavoidable.
However, knowing that (\ref{eq:binom-skel}) is collectively $r$-unavoidable if
$m=\sum_{i=1}^r m_i+r-1$, it is sufficient to check the condition
\begin{equation}
   m = N+1 = s(k+2) + (r-s)(k+1) + r-1 \, .
\end{equation}
 Since $rk+s = (r-1)d$ this is equivalent to $N = (r-1)(d+2)$.
\end{proof}

\subsection{Example}\label{sec:example}

Let $m= 2k+2$ and assume that $K$ is a {\em $(m,k)$-balanced} simplicial subcomplex of $2^{[m]}$.
Let $K^\circ$ be the Alexander dual of $K$. Then the pair $(K, K^\circ)$ of simplicial complexes is
collectively $2$-unavoidable (by the definition of Alexander duality, see also \cite{jnpz}). Moreover, $K^\circ$ is also $(m,k)$-balanced.
By Theorem~\ref{thm:main} the symmetric deleted join,
\begin{equation}\label{eqn:SymmDelJoin}
    SymmDelJoin(K, K^\circ)
\end{equation}
is an $(m-3)$-connected, $(m-2)$-dimensional simplicial complex.

\begin{rem}{\rm
The symmetric deleted join (\ref{eqn:SymmDelJoin}) is a union of two overlapping Bier spheres of dimension $(m-2)$. It follows (essentially from the Mayer-Vietoris exact sequence)
that $(K\ast K^\circ)\cup (K^\circ\ast K)$  is $(m-3)$-connected if and only if $(K\ast K^\circ)\cap (K^\circ\ast K)$ is $(m-4)$-connected. This holds for balanced complexes and can be
 established by a direct argument.
}
\end{rem}

As a consequence we obtain a result which generalizes the classical Van Kampen-Flores theorem and reduces to the Theorem~6.6 from \cite{bfz} if $K = {{[m]}\choose{\leqslant k}}$.

\begin{thm}
For each continuous map $f : \Delta^{m-1} \rightarrow \mathbb{R}^{m-2}$ (where $m = 2k+2$) and each $(m,k)$-balanced simplicial complex $K\subset 2^{[m]}$, there
exist faces $F_1\in K$ and $F_2\in K^\circ$ such that
\[
  f(F_1) \cap f(F_2) \neq \emptyset \, .
\]
\end{thm}

\subsection{An example and a comparison with earlier results}

Here we give an example which explains why Theorem~\ref{thm:main} (and its consequence Theorem~\ref{thm:main-2}) are not expected to be
immediate consequences of known results. In other words they cannot be reduced to the case when all the complexes are skeleta, either ${[m]}\choose{\leqslant k}$  or ${[m]}\choose{\leqslant k +1}$.

\medskip
Let $m=9, \ k=2$. Let $K_1= {[9]\choose{\leqslant 2}}\cup ({[9]\choose{\leqslant 2}}\setminus \{A\})$ be the complex  on 9 vertices $1,...,9$ containing all $2$-element subsets and all  $3$-element sets except  $A=\{789\}$.
Let $\Delta$ be an Alexander self-dual complex on the vertices $1,...,6$, for example the minimal, $6$-vertex triangulation of the real projective plane.
Let $K_2=K_3$ be the complex  on 9 vertices $1,...,9$ containing all  $2$-element subsets together with all $3$-element subsets that belong to $\Delta$.

\begin{lemma} Both the triple $\mathcal{K} = (K_1,K_2,K_3)$ and the triple of skeleta
$$ \mathcal{L} = (L_1, L_2, L_3) = \big( {[9]\choose \leqslant 3}, {[9]\choose \leqslant 2}, {[9]\choose \leqslant 2}  \big)    $$ are collectively unavoidable.
\end{lemma}

\medskip
Let us observe that the symmetrized deleted join $SymmDelJoin(\mathcal{L})$ is not contained in $SymmDelJoin(\mathcal{K})$.
For example, $$(789,34,12;56)\in SymmDelJoin(\mathcal{L})\setminus SymmDelJoin(\mathcal{K})\, .$$

\medskip
This fact indicates that there does not exist an obvious $S_3$-equivariant map  $f : SymmDelJoin(\mathcal{L})\rightarrow SymmDelJoin(\mathcal{K})$
and therefore this argument cannot be applied to deduce Theorem~\ref{thm:main-2} from the
``Balanced Van Kampen-Flores theorem'' (Theorem~\ref{thm:glavna-jvz}).

\section{Collective unavoidability of balanced $r$-tuples}
\label{sec:collective}

Let $K_1,\dots, K_r$ be a collection of $(m,k)$-balanced complexes.
Each $K_i$ can be represented as
\[
K_i =  {{[m]}\choose{\leqslant k+1}} \setminus \mathcal{A}^i = {{[m]}\choose{\leqslant k+1}} \setminus \{A^i_1,\dots, A^i_{k_i}\}
\]
where $\vert A^i_j\vert = k+1$ for each $i$ and $j$.

Let us define an $r$-partite graph $\Gamma = \Gamma(K_1,\dots, K_r)$ whose vertices are
labeled by pairs $(i,j)$ where $i=1,\dots, r$ and $j = 1,\dots, k_i$ for each $i$. Two vertices $(i,j)$ and $(i',j')$ share an edge if $i\neq i'$
and $A^i_j\cap A^{i'}_{j'}=\emptyset$.

\begin{rem} {\rm
In agreement with the standard definition of the Kneser graph $KG(\mathcal{F})$ of a family of sets (see ...) the graph $\Gamma$ can be also
described as the $r$-partite Kneser graph $KG(\mathcal{A})$ of the graded family $\mathcal{A} = \mathcal{A}^1\sqcup\dots\sqcup\mathcal{A}^r$. }
\end{rem}

\begin{prop}\label{prop:KG}
Assume that $d = r(k+2)-n$. Suppose that $K_1,\dots, K_r$ is a collection of $(m,k)$-balanced subcomplexes of $2^{[m]}$
 and let $\Gamma = \Gamma(K_1,\dots, K_r)$ be the associated $r$-partite Kneser graph.

\begin{enumerate}
\item[(1)]  If $d>r$, then the collection $(K_1,\dots, K_r)$ is {\em always} collectively unavoidable.

\item[(2)] If $d<1$,   then the collection $(K_1,\dots, K_r)$ is {\em not} collectively unavoidable.

\item[(3)]  If $d=1$ the only collectively unavoidable $r$-tuple is $$K_1 = K_2 = \dots = K_r = {{[m]}\choose{\leqslant k+1}}\, .$$

\item[(4)] If $1<d\leq r$, then the collection $(K_1,\dots, K_r)$ is collectively unavoidable iff the $r$-partite associated Kneser graph $\Gamma = \Gamma(K_1,\dots, K_r)$
          contains no $d$-clique.
\end{enumerate}
\end{prop}

\proof

\par\noindent (1)\, In this case $d>r \Leftrightarrow r(k+2)-m > r \Leftrightarrow r(k+1)>m$.
For each partition $[m] = B_1\sqcup\dots\sqcup B_r$ there exists $i$ such that  $\vert B_i\vert \leq m$, hence $B_i\in K_i$.

\par\noindent (2)\, If $d<1$ then $m>r(k+2)$. In this case there exists a partition $[m] = B_1\sqcup\dots\sqcup B_r$ where $\vert B_i\vert \geq k+2$ for each $i$, proving that
$(K_1,\dots, K_r)$ is {\em not} collectively unavoidable.

\par\noindent (3)\, In this case $m = r(k+2)-1$. Let $B\in {{[m]}\choose{\leqslant k+1}}$ be an arbitrary set. For a given $i$ let
$[m] = B_1\sqcup\dots\sqcup B_r$ be a partition where $B_i = B$ and $\vert B_j\vert = k+2$ for $j\neq i$. If the $r$-tuple $(K_1,\dots, K_r)$ is collectively unavoidable
then $B=B_i\in K_i$ which implies that $K_i =   {{[m]}\choose{\leqslant k+1}}$.

 \par\noindent (4) \, Assume that there is a $d$-clique in the graph $\Gamma = \Gamma(K_1,\dots, K_r)$, associated to a collectively $r$-unavoidable $r$-tuple $K_1,\dots, K_r$.
 Assume that the vertices of the $d$-clique are $A^1_1, A^2_1,\dots, A^d_1$. Let us choose a partition
 $$[m] = A^1_1 \sqcup A^2_1 \sqcup \dots \sqcup A^d_1 \sqcup B_1\sqcup \dots  \sqcup B_{r-d}$$
 such that $\vert B_i\vert = k+2$ for each $i$. This is possible since $d(k+1)+(r-d)(k+2)$ adds up to $m$. This is in contradiction with the collective unavoidability of the
 collection $\langle K_i \rangle_{i=1}^r$ since,
 \[
 A^1_1 \notin K_1,  A^2_1 \notin K_2, \dots, A^d_1 \notin K_d, B_1\notin K_{d+1}, \dots  B_{r-d} \notin K_r \, .
 \]
 For the opposite direction, assume that there exists a partition  $[m] = B_1\sqcup\dots\sqcup B_r$ such that
 $B_i\notin K_i$ for each $i=1,\dots, r$. It follows that $\vert B_i\vert \geq k+1$ for each $i$ and, since $m = r(k+2)-d$, at least $d$ of them have
 exactly $k+1$ elements. This is a clique with desired properties. \hfill $\square$

\section*{Appendix. Discrete Morse theory}
\label{sec:appendix}

By definition, a discrete Morse function (a DMF, for short) is an acyclic
matching on the Hasse diagram of a simplicial complex.
In more details,assume we have a simplicial complex. By $\alpha^p, \ \beta^p$ we
denote its $p$-dimensional simplices, or \textit{$p$-simplices}, for short.
A \textit{discrete vector field} is a matching
$$\big(\alpha^p,\beta^{p+1}\big)$$
such that:
\begin{enumerate}
\item each simplex of the complex is matched at most once,
and
\item in each matched pair, the simplex $\alpha^p$ is a facet of
$\beta^{p+1}$.\end{enumerate}
Given a discrete vector field, a \textit{gradient path} is a sequence of
simplices
$$\alpha_0^p, \ \beta_0^{p+1},\ \alpha_1^p,\ \beta_1^{p+1}, \ \alpha_2^p,\
\beta_2^{p+1} ,..., \alpha_k^p,\ \beta_k^{p+1},\ \alpha_{k+1}^p,$$
which satisfies the conditions:
\begin{enumerate}
\item $\alpha_i^p$ and $\beta_i^{p+1}$ are matched.
\item Whenever $\alpha$ and $\beta$ are neighbors in the path,
$\alpha$ is a facet of $\beta$.
\item $\alpha_i\neq \alpha_{i+1}$.
\end{enumerate}
A path is a \textit{closed} if $\alpha_{k+1}^p=\alpha_{0}^p$.

\medskip
 Along a gradient path we distinguish the ``matching steps'' $\alpha^p_i \nearrow \beta^p_{i}$ from the
  ``splitting steps'' $\beta^p_{i} \searrow \alpha^p_{i+1} $.

\medskip
A
\textit{discrete Morse function on a simplicial complex} is a
discrete vector field without closed paths.
Assuming that a discrete Morse function is fixed, the \textit{
critical simplices} are the non-matched simplices.
\medskip
A DMF gives a way of contracting all the
simplices of the complex that are matched: if a simplex $\beta$ is matched with its
facet $\alpha$
then these two can be contracted by pushing $\alpha$
inside $\beta$.
Acyclicity guarantees that if we have many matchings at a time, one
can consequently perform the contractions. The order of contractions
does not matter, and eventually one arrives at a complex homotopy
equivalent to the initial one.

\medskip
In the paper we use the following fact (it follows straightforwardly from the
above): if a simplicial complex has one zero-dimensional critical simplex, and the
dimensions of the other critical simplices are greater than some $N$, then
the complex is $(N-1)$-connected.

\end{document}